\documentclass[runningheads]{llncs}
\usepackage{graphicx}
\usepackage{amsmath}
\usepackage{amsfonts}
\usepackage{amssymb}
\begin{document}
\title{Sufficient Conditions for the Joined Set of Solutions of the Overdetermined Interval System of Linear Algebraic Equations Membership to Only One Orthant}
\titlerunning{On the membership of the admissible set ISLAE to a single orthant}
%
\author{Vladimir Erokhin\inst{1}\orcidID{0000-0002-1760-7859} \and
Vitaly Kakaev\inst{1}\orcidID{0000-0001-9898-2740} \and
Andrey Kadochnikov\inst{1}\orcidID{0000-0002-2142-3681} \and
Sergey Sotnikov\inst{1}\orcidID{0000-0002-9712-4717}}
\authorrunning{V. Erokhin et al.}
\institute{
Mozhaisky Military Space Academy, 13, Zhdanovskaya Street, \\
197198 St. Petersburg, Russia \\
\email{vka@mil.ru}
}
\maketitle              
\begin{abstract}
Interval systems of linear algebraic equations (ISLAE) are considered in the context
of constructing  of linear models according to data with interval uncertainty.
Sufficient conditions for boundedness and convexity of an admissible domain (AD) of ISLAE and its belonging to only one orthant of an $n$-dimensional space are proposed, which can be verified in polynomial time by the methods of computational linear algebra.
In this case, AD ISLAE turns out to be a convex bounded polyhedron, entirely lying in the corresponding ortant.
These properties of AD ISLAE allow, firstly, to find solutions to the corresponding ISLAE in polynomial time by linear programming methods (while finding a solution to ISLAE of a general form is an NP-hard problem).
Secondly, the coefficients of the linear model obtained by solving the corresponding ISLAE have an analogue of the significance property of the coefficient of the linear model, since the coefficients of the linear model do not change their sign within the limits of the AD.
The formulation and proof of the corresponding theorem are presented.
The error estimation and convergence of an arbitrary solution  of ISLAE to the normal solution of a hypothetical exact system of linear algebraic equations are also investigated.
An illustrative numerical example is given.

\keywords
\keywords{Interval systems of linear algebraic equations  
\and 
Polynomial-time solvability  
\and 
Convergence 
\and
Error estimation 
\and
Analog of statistical significance property.
}
\end{abstract}
%

%
\section{Introduction}
Overdetermined interval systems of linear algebraic equations are a native tool for creating data processing algorithms with interval uncertainty and estimating the parameters of the corresponding linear models \cite{Voshchinin,Belov,Polyak,Zhilin,Madiyarov,Shary20}.
These systems can be defined as follows
\begin{gather}
\label{eq1}
Ax=b,\;
\underset{\raise0.3em\hbox{$\smash{\scriptstyle-}$}}{A}\leqslant A\leqslant \bar A,\;
\underset{\raise0.3em\hbox{$\smash{\scriptscriptstyle-}$}}{b} \leqslant b \leqslant \bar b,\;
\end{gather}
where $\underset{\raise0.3em\hbox{$\smash{\scriptstyle-}$}}{A} ,\bar A \in {\mathbb{R}^{m \times n}}$ are given matrices;
$\underset{\raise0.3em\hbox{$\smash{\scriptstyle-}$}}{b} ,\bar b \in {\mathbb{R}^m}$ are given vectors,
such that
$\underset{\raise0.3em\hbox{$\smash{\scriptscriptstyle-}$}}{A} \leqslant \bar A$,
$\underset{\raise0.3em\hbox{$\smash{\scriptscriptstyle-}$}}{b} \leqslant \bar b$;
$A =(a_{ij})\in \mathbb{R}^{m \times n}$,
${x =(x_j)\in {\mathbb{R}^n}}$,
${b =(b_j)\in {\mathbb{R}^m}}$ are unknown (to be determined) matrix and vectors,
$\underset{\raise0.3em\hbox{$\smash{\scriptscriptstyle-}$}}{A} \ne \bar A$,
$\underset{\raise0.3em\hbox{$\smash{\scriptscriptstyle-}$}}{b} \ne \bar b$,
$m>n$.

In the study of ISLAE, the focus is often only on the so-called \textit{joined set of solutions} ISLAE \cite{Shary20}, defined as
$$
{\mathbf{X}} = \left\{ x\left|
\left(
\exists
A,b
\left|
\underset{\raise0.3em\hbox{$\smash{\scriptscriptstyle-}$}}{A} \leqslant A \leqslant \bar A,\underset{\raise0.3em\hbox{$\smash{\scriptscriptstyle- }$}}{b} \leqslant b \leqslant \bar b,Ax = b \right.
\right)
\right.
\right\}.
$$

The equivalent \eqref{eq1} representation of ISLAE can be written using the \textit{middle matrix}
${A_c} = (a_{ij}^c)=\frac{1}{2}( {\underset{\raise0.3em\hbox{$\smash{\scriptscriptstyle-}$}}{A} + \bar A} )$,
\textit{radius matrix}
${A_r} = (a_{ij}^r)=\frac{1}{2}( {\bar A - \underset{\raise0.3em\hbox{$\smash{\scriptscriptstyle-}$}}{ A} } )$,
\textit{middle vector}
${b_c} = (b_i^c)=\frac{1}{2}( {\underset{\raise0.3em\hbox{$\smash{\scriptscriptstyle-}$}}{b} + \bar b} )$
and \textit{radius vector}
${b_r} = (b_i^c)=\frac{1}{2}( {\bar b - \underset{\raise0.3em\hbox{$\smash{\scriptstyle-}$}}{b} } )$:
\begin{equation}
\label{eq1a}
Ax=b,A_c-A_r\leqslant A \leqslant A_c+A_r, b_c-b_r\leqslant b \leqslant b_c+b_r.
\end{equation}

In terms of these vectors and matrices, an important result is usually formulated that characterizes the set $\mathbf{X}$ and any admissible solution $\{A,b,x\}$ of ISLAE.
Consider the nonlinear system of inequalities
\begin{equation}
\label{OP}
\left| {{A_c}x - {b_c}} \right| \leqslant {A_r}\left| x\right| + {b_r},
\end{equation}
where $\left| \cdot \right|$ is an element-by-element operation of taking an absolute value.
Denote by the symbol
${\mathbf{\overset{\lower0.5em\hbox{$\smash{\scriptscriptstyle\frown}$}}{X} }}$
set of system \eqref{OP} solutions.

Occurs

\begin{theorem}[Oettli-Prager Theorem \cite{Oettli-Prager-64}]
\label{TOP}
\begin{equation}
\label{eq3}
\mathbf{X} \equiv \bf{\overset{\lower0.5em\hbox{$\smash{\scriptscriptstyle\frown}$}}{X} }.
\end{equation}
Moreover, if $x$ is a solution to the system of inequalities \eqref{OP}, matrix $A$ and vector $b$,
satisfying the conditions \eqref{eq1} or, equivalently, \eqref{eq1a},
can be built according to the formulas
\begin{gather*}
a_{ij}=a_{ij}^c+\Delta a_{ij},
b_i=b_i^c+\Delta b_i, \\
\Delta a_{ij} = -d_i a_{ij}^r \operatorname{sign}(x_j)/\gamma_i,
\Delta b_i = -d_i b_i^r /\gamma_i,\\
d=(d_i)=A_c x-b_c,
\gamma_i=\sum\limits_{i=1}^n {a_{ij}^r\left| x_j\right|+b_i^r}.
\end{gather*}
\end{theorem}

It follows from the  Theorem~\ref{TOP} that the admissible set ${\mathbf{\overset{\lower0.5em\hbox{$\smash{\scriptscriptstyle\frown}$}}{X} }}$
systems of inequalities \eqref{OP}, and hence the joint set of ISLAE solutions
$\mathbf{X}$,
are the union of admissible sets $2^n$ of systems of linear inequalities,
each of which lies in one of the $2^n$ orthants of the $n$-dimensional space.
However, this set
${\mathbf{\overset{\lower0.5em\hbox{$\smash{\scriptscriptstyle\frown}$}}{X} }}$
($\mathbf{X}$) need not be convex, connected, and bounded.

This fact is a convincing illustration of the NP-complexity of the problem of finding solutions to ISLAE in the general case \cite{Ron}.
\vskip -6mm
\begin{figure}
\begin{tabular}{cc}
\includegraphics[width=0.4\textwidth]{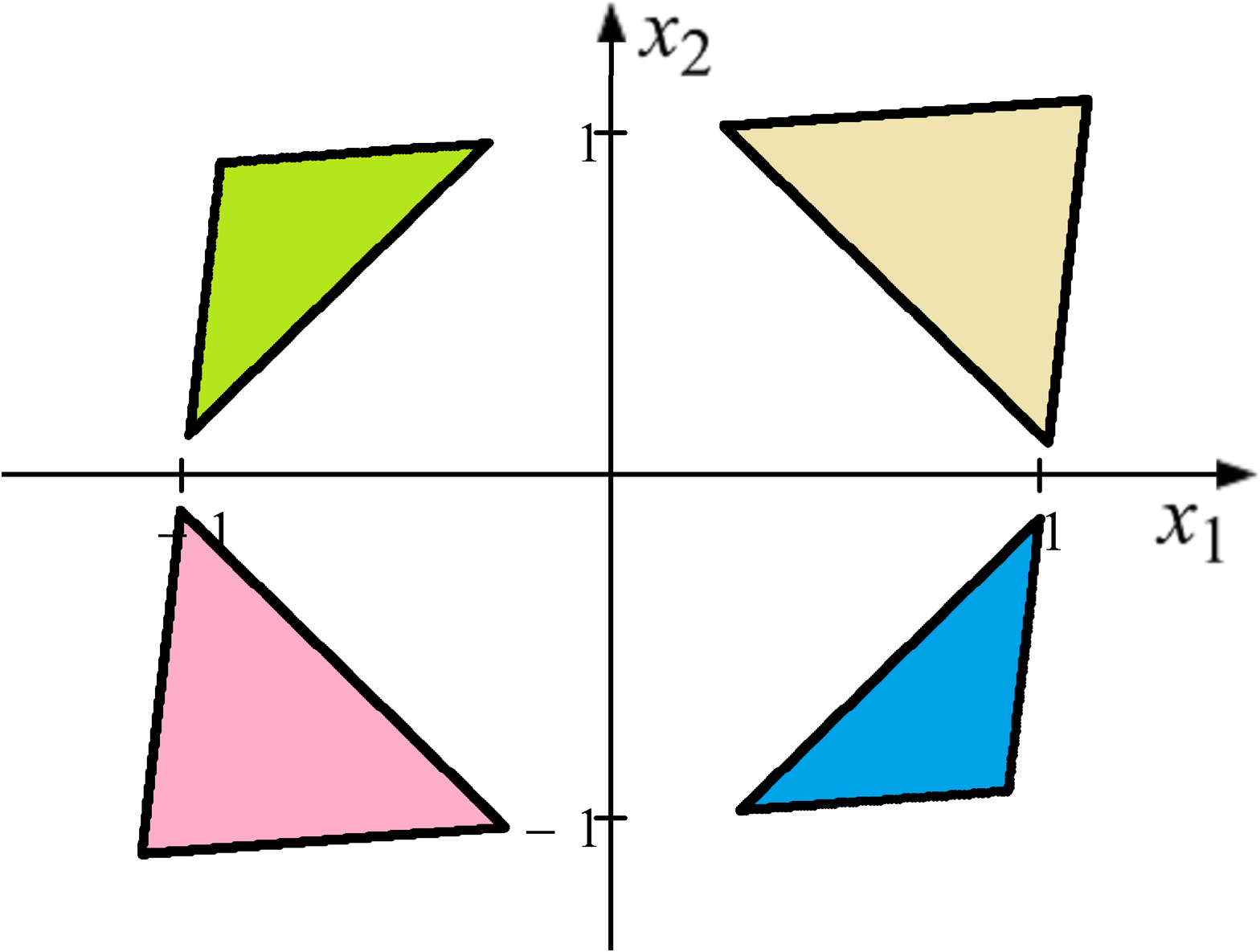}
& 
\includegraphics[width=0.5\textwidth]{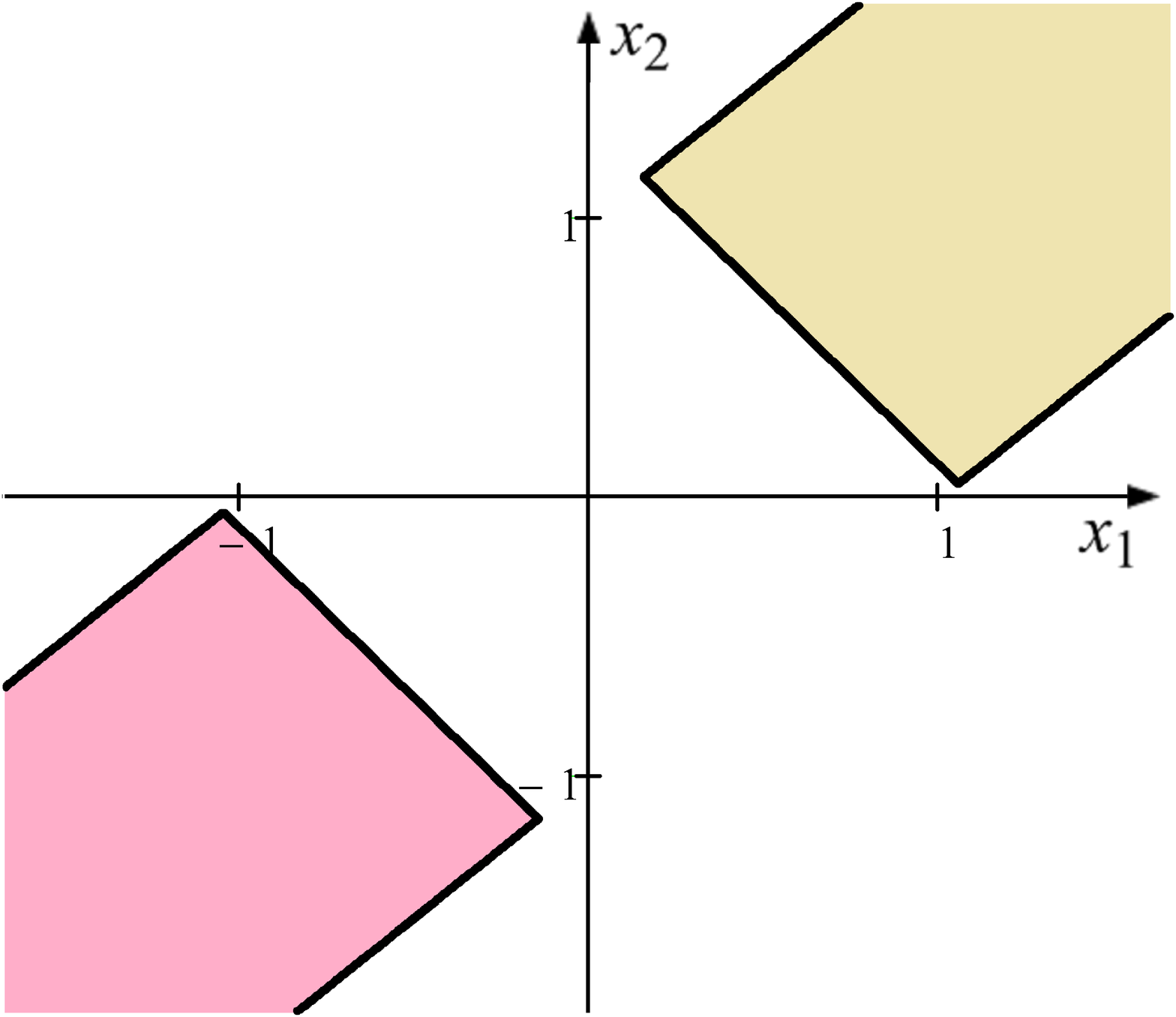}
\\
${A_c} = \left( {\begin{smallmatrix}
  {1.1}&{ - 0.1} \\ 
  {0.1}&{ - 1.1} \\ 
  0&0 
\end{smallmatrix}} \right)$,
${b_c} = \left( {\begin{smallmatrix}
  0 \\ 
  0 \\ 
  {1.2} 
\end{smallmatrix}} \right)$,
&
${A_c} = \left( {\begin{smallmatrix}
  1&{ - 1} \\ 
  {1.1}&{ - 1.1} \\ 
  0&0 
\end{smallmatrix}} \right)$,
${b_c} = \left( {\begin{smallmatrix}
  {0} \\ 
  {0} \\ 
  {1.2}
\end{smallmatrix}} \right)$,
\\
${A_r} = \left( {\begin{smallmatrix}
  0&0 \\ 
  0&0 \\ 
  {1.1}&{0.9} 
\end{smallmatrix}} \right)$,
${b_r} = \left( {\begin{smallmatrix}
  {1.1} \\ 
  {1.1} \\ 
  {0} 
\end{smallmatrix}} \right)$.
&
${A_r} = \left( {\begin{smallmatrix}
  0&0 \\ 
  0&0 \\ 
  {1.1}&{0.9} 
\end{smallmatrix}} \right)$,
${b_r} = \left( {\begin{smallmatrix}
  {1.1} \\ 
  {1.1} \\ 
  {0} 
\end{smallmatrix}} \right)$.
\\
 &
\\
{a) $\mathbf{X}$ is a disconnected bounded domain} &
{b) $\mathbf{X}$ is disconnected unbounded domain}
\end{tabular} 
\caption{Examples of ISLAEs with disconnected (non-convex) joined sets of solutions}
\end{figure}

The NP-hardness of finding ISLAE solutions in the general case is a limiting factor for the introduction of this toolkit into the practice of modeling and data analysis.
At the same time, as often shown by the study of redefined ISLAEs related to the solution of practical (engineering) problems of constructing linear dependencies from experimental data (with interval uncertainty), the admissible set of ISLAE 1) is a convex polyhedron that lies entirely in some orthant \mbox{$n$-dimensional} space and 2) as the number of experiments increases, it contracts to a point coinciding with the true vector of coefficients of the linear model.

Property 2) is analogous to \emph{consistency} (see, for example, \cite{Ibragimov,Seber}) of a statistical model.
Property 1) firstly, guarantees \emph{polynomial complexity} of finding \mbox{ISLAEs} solutions (using linear programming methods, see e.g. \cite{Schrijver}), and secondly, is an analog of the \emph{significance} coefficients of the statistical linear model \cite{Seber}.

\section{Spade-work}

Let us introduce the definitions necessary for subsequent calculations and justify auxiliary results.

Let us introduce the definitions necessary for subsequent calculations and justify auxiliary results.

Let $\hat x = (\hat x_j) = A_c^ + {b_c}$ be the normal least squares pseudo-solution (LS-solution) of the inconsistent overdetermined system of linear algebraic equations (SLAE) ${A_c}x \cong {b_c }$,
$\Delta b_c = {b_c} - {A_c}\hat x$ is its discrepancy with the minimum Euclidean norm,
$A_c^ + $ is the corresponding pseudo-inverse matrix, and the conditions
${A_c^1\hat x \leqslant b_c^1}$,
${-A_c^2}\hat x \leqslant {-b_c^2}$,
where
up to some permutation of strings $A_c$ and elements $b_c$
\[
{A_c} = \left[ {\begin{array}{*{20}{c}}
  {A_c^1} \\ 
  {A_c^2} 
\end{array}} \right],\;{b_c} = \left[ {\begin{array}{*{20}{c}}
  {b_c^1} \\ 
  {b_c^2} 
\end{array}} \right].
\]  

Let us introduce the notation:
\begin{gather*}
{{\tilde A}_c}  =  \left[ {\begin{array}{*{20}{c}}
  {A_c^1} \\ 
  { - A_c^2} 
\end{array}} \right],\;{{\tilde b}_c}  =  \left[ {\begin{array}{*{20}{c}}
  {b_c^1} \\ 
  { - b_c^2} 
\end{array}} \right],\;
S  =  \mathrm{diag} \left( {\mathrm{sign} \left( {\hat x} \right)} \right),\\
\overset{\lower0.5em\hbox{$\smash{\scriptscriptstyle\frown}$}}{x} =
(\tilde A_c - A_r S)^+ (\tilde b_c + b_r),\;
\Delta \overset{\lower0.5em\hbox{$\smash{\scriptscriptstyle\frown}$}}{b}
= (\tilde b_c + b_r) - (\tilde A_c - A_r S)\overset{\lower0.5em\hbox{$\smash{\scriptscriptstyle\frown}$}}{x}, \\
\tilde{\mathbf{X}} =
\left\{ {x\left| {({A_c} - {A_r}S)x \leqslant {b_c} + {b_r},
\;
( - {A_c} - {A_r}S)x \leqslant - {b_c} + {b_r}} \right.} \right\},
\\
{1} \text{ is a } n\text{-dimensional vector of ones}, \\
0_n \text{ is a zero matrix of order } n, \\
I_n \text{ is the identity matrix of order } n, \\
\sigma_{\min }^{{A_c}}
\text{ is the minimum singular value of the matrix } A_c\;, \\
\sigma _{\max }^{{A_r}}
\text{ is the maximum singular value of the matrix } A_r, \\
||\cdot|| \text{ this is the Euclidean vector or spectral matrix norm} \\
\text{ (depending on the context), }
\end{gather*}
\vskip -4mm
function
$\operatorname{sign}(\cdot)$ applied to vector argument
$\hat x$
element by element, returning an $n\text{-dimensional}$ vector composed of the numbers $\{-1,0,+1\}$
according to the signs of the elements $\hat x_j$.

\vskip 3mm

Note that
$1/\sigma _{\min }^{{A_c}}=||A_c^+||$,
$\sigma _{\max }^{{A_r}}=||A_r||$, and similar relations are also valid for other matrices encountered in the text and pseudoinverses to them (see, for example, \cite{Lawson,Horn}).

The following lemmas are valid. 
\begin{lemma}
\label{L_1}
Systems of linear inequalities
\begin{equation}
\label{eq8}
({{\tilde A}_c} - {A_r}S)x \leqslant {{\tilde b}_c} + {b_r}\;,\; Sx\geqslant 0
\Leftrightarrow
\left[ {\begin{array}{*{20}{c}}
  \tilde A_c - A_r S \\ 
  { - S} 
\end{array}} \right]x \leqslant \left[ {\begin{array}{*{20}{c}}
  \tilde b_c + b_r \\ 
  0 
\end{array}} \right]
\end{equation}
and
\begin{equation}
\label{IE1}
({\tilde A_c} - {A_r}S)x \leqslant {\tilde b_c} + {b_r}
\end{equation}
are consistent.
\end{lemma}
\begin{proof}
Taking into account the above definitions of objects 
$\hat x$, $\tilde A_c$, $\tilde b_c$, $S$
and given the conditions $A_r\geqslant 0$, $b_r\geqslant 0$,
it is easy to make sure that the vector 
$\hat x$
belongs to the set of valid solutions of \eqref{eq8} and \eqref{IE1} systems.
\end{proof}
\begin{lemma}
\label{L_3}
If the system of linear inequalities
\[
{Ax\leqslant b},
  {Sx \geqslant 0}, 
\]
where $A\in \mathbb{R}^{m\times n}$,
$b\in \mathbb{R}^{m}$,
$x\in\mathbb{R}^{mn}$,
is consistent, and
the condition is met
\begin{equation}
\label{eq14}
\forall x\left| {Ax \leqslant b,Sx \geqslant 0 \Rightarrow Sx \geqslant {{1}}\delta ,} \right.
\end{equation}
where 
$S$ is a diagonal matrix of order $n$ with elements $s_j = \pm 1$ on the diagonal,
$\delta>0$ is some scalar,
then
the relation is valid
\begin{equation}
\label{eq15}
\forall x\left| Ax \leqslant b \Rightarrow Sx \geqslant {{1}}\delta.\right.
\end{equation}
\end{lemma}
\begin{proof}
Assume the opposite: let there be a vector $y\in \mathbb{R}^n$ such that
$Ay\leqslant b$,
$s_j y_j \leqslant 0$,
$s_k y_k \geqslant 0$,
where $j\in\{1,\ldots,n\}$ is some index,
$k=1,2,\ldots,j-1,j+1,\ldots,n$.
Moreover, let $z$ be a vector such that
$Az \leqslant b$, $Sz \geqslant {{1}}\delta$. Consider also
$x(\alpha ) = \alpha y + (1 - \alpha )z$.
Due to the convexity of the admissible domain of any system of linear inequalities
the vector $x(\alpha)$ belongs to the admissible region of the system $Ax\leqslant b$ for any $0\leqslant\alpha\leqslant 1$.
It is easy to show that the specified constraints are satisfied by the parameter $\tilde \alpha$ such that $x_j(\tilde \alpha)=0$.
In this case, the conditions $Ax(\tilde \alpha)\leqslant b$, $Sx(\tilde \alpha)\geqslant 0$ are satisfied.
Therefore, due to \eqref{eq14},
$Sx(\tilde \alpha)\geqslant{{1}}\delta$,
which contradicts the condition $x_j(\tilde \alpha)=0$.
\end{proof}
\begin{lemma}
\label{L_4}
If the system of inequalities $Ax\leqslant b$ is compatible and the condition \eqref{eq15} is satisfied,
then for any joint system of linear inequalities of the form
${Ax \leqslant b}$, ${Cx \leqslant d}$, where $C$ and $d$ are arbitrary matrix and vector with matching between themselves and vector $x$
dimensions, then the corollary
\begin{equation*}
\forall x\left| {Ax \leqslant b, Cx \leqslant d \Rightarrow Sx \geqslant {{1}}\delta} \right..
\end{equation*}
\end{lemma}
\begin{proof}
The assertion of the lemma follows directly from the Minkowski-Farkash lemma on consequences \cite[Theorem 4.7]{Ashmanov}.
\end{proof}

\begin{theorem}[\cite{Lawson}]
\label{TL}
Let $x$ be a normal pseudo-solution of the 
\textquotedblleft
exact\textquotedblright
~least squares problem
$Ax\cong b$,
where
$A\in\mathbb{R}^{m\times n}$,
$b\in\mathbb{R}^m$,
$\operatorname{rank}A=n$,
$r = b - Ax$ is corresponding residual vector;
let be
$x + \Delta x$ is normal pseudosolution 
\textquotedblleft
perturbed\textquotedblright
~least squares problems
$(A + \Delta A)(x + \Delta x)\cong (b+\Delta b)$,
where
$\Delta A\in\mathbb{R}^{m\times n}$,
$\Delta b\in\mathbb{R}^m$,
and the conditions are met
$\operatorname{rank}A = n < m$,
$\sigma _{\min }^{A}>\sigma _{\max }^{\Delta A}$.

Then
\begin{gather*}
\operatorname{rank}(A +\Delta A)=n,\\
||\Delta x|| \leqslant
\frac{1}{\sigma_{\min }^{A} - \sigma_{\max }^{\Delta A}}
\left(
\sigma_{\max }^{\Delta A}
\left(
||\hat x||
+
\frac{||r||}{\sigma_{\min }^{A}}
\right)
+
||\Delta b||
\right).
\end{gather*}
\end{theorem}

\section{Main result}

\begin{theorem}
\label{TT2}
Let the conditions be fulfilled
\begin{gather}
\label{E1}
\operatorname{rank}A_c=n,\; \sigma _{\min }^{A_c}>\sigma _{\max }^{A_r}, \\
\label{eq4}
A_r S \hat x\leqslant b_r, \\
\label{E2}
\mathop {\min }\limits_{j = 1,\ldots,n} \left| \hat x_j \right|
>\!
\left(
\gamma =
\frac{1}{\sigma _{\min }^{A_c} - \sigma _{\max }^{A_r}}\left( 
\sigma _{\max }^{A_r}
\left(
||\hat x||
+
\frac{||\Delta b_c||}{\sigma _{\min }^{A_c}}
\right)
+
||b_r||
\right) 
\right)
\!> 0, \\
\label{E3}
\Delta \overset{\lower0.5em\hbox{$\smash{\scriptscriptstyle\frown}$}}{b} > 0, \\
\label{E4}
||\Delta \overset{\lower0.5em\hbox{$\smash{\scriptscriptstyle\frown}$}}{b}||^2\cdot
\mathop {\max }\limits_{i,j} {q_{ij}} < 1,
\end{gather}
where $q_{ij}$ is an element of the matrix
\begin{equation}
\label{E5}
Q = (q_{ij}) = 
\operatorname{diag}
(S\overset{\lower0.5em\hbox{$\smash{\scriptscriptstyle\frown}$}}{x})^{-1}
\cdot
(\tilde A_c - A_r S)^+
\cdot
\operatorname{diag}
(\Delta \overset{\lower0.5em\hbox{$\smash{\scriptscriptstyle\frown}$}}{b})^{-1}.
\end{equation}
Then 
\begin{enumerate}
\item
Admissible domains of systems of linear inequalities
\eqref{eq8} and \eqref{IE1}
are not empty bounded convex polyhedra.
\item
There is such a number $\delta>0$ that
the condition is true
\begin{equation}
\label{eq10}
\forall x\left|
(\tilde A_c - A_r S)x\leqslant \tilde b_c+b_r
\right.
\Rightarrow
Sx \geqslant {1}\delta .
\end{equation}  
\item All $2^n$ systems of linear inequalities 
\begin{equation}
\label{eq11}
({{\tilde A}_c} - {A_r}\tilde S)x \leqslant {{\tilde b}_c} + {b_r},
\end{equation}
where ${\tilde S}$ is a diagonal matrix of order $n$ with elements $\pm 1$ on the diagonal, are compatible.
In this case, the system of linear inequalities 
$Sx\geqslant{1}\delta$
is a consequence of any of them. 
\item
Set 
${\mathbf{X}}$ coincides with the set
$\tilde{\mathbf{X}}$
and, in the case of non-emptiness, is a convex bounded polyhedron lying strictly inside the orthant defined by 
the signs of the diagonal elements of the matrix $S$ or, equivalently, the signs of the elements of the vector $\hat x$ 
MNC solutions of SLAE ${A_c}x\cong{b_c}$.
\end{enumerate}
\end{theorem}
\begin{proof}
\begin{enumerate}
\item
By virtue of the lemma \ref{L_1}, the systems of linear inequalities \eqref{eq8}
and
\eqref{IE1}
are compatible (the corresponding valid domains \emph{are not empty}). 

By virtue of the condition \eqref{E2}, in the formulation of which
, by virtue of the theorem \ref{TL}
$\gamma$ is the upper bound 
$||\Delta x||=|| \hat x - \overset{\lower0.5em\hbox{$\smash{\scriptscriptstyle\frown}$}}{x} ||$ error of the LS-solution of the perturbed SLAE
$(\tilde A_c - A_r S)(\hat x + \Delta x) \cong \tilde b_c + b_r$,
the conditions are met 
$S\hat x>0$,
$S\overset{\lower0.5em\hbox{$\smash{\scriptscriptstyle\frown}$}}{x}>0$. By virtue of the last condition and assumption
\eqref{E3} the condition is true 
$
S\overset{\lower0.5em\hbox{$\smash{\scriptscriptstyle\frown}$}}{x} 
\overset{\lower0.5em\hbox{$\smash{\scriptscriptstyle\frown}$}}{\Delta b^\top} > 0
$
.
Let's build two $(n\times(m+n))$-matrices as follows:
\begin{gather*}
\label{P}
P = 
\left[ {\begin{array}{*{20}{c}}
  {{P_1}}&{{P_2}} 
\end{array}} \right]
=
\left[ 
{\begin{array}{*{20}{c}}
  {-(\tilde A_c - A_r S)^+ + \alpha S \overset{\lower0.5em\hbox{$\smash{\scriptscriptstyle\frown}$}}{x} }
\overset{\lower0.5em\hbox{$\smash{\scriptscriptstyle\frown}$}}{\Delta b^\top}
&
{0_n} 
\end{array}} 
\right]
,\\
\label{Q}
Q = 
\left[ {\begin{array}{*{20}{c}}
  {{Q_1}}&{{Q_2}} 
\end{array}} \right]
=
\left[ 
{\begin{array}{*{20}{c}}
  {(\tilde A_c - A_r S)^+ + \beta S \overset{\lower0.5em\hbox{$\smash{\scriptscriptstyle\frown}$}}{x} }
\overset{\lower0.5em\hbox{$\smash{\scriptscriptstyle\frown}$}}{\Delta b^\top}
&
{0_n} 
\end{array}} 
\right]
,
\end{gather*}
where $\alpha, \beta >0$ are some scalar parameters. Let's choose the values of the specified parameters so
that the conditions are met 
\begin{equation}
\label{D0}
P,Q\geqslant 0.
\end{equation} 

Since $S$ is an orthogonal matrix, due to the properties of the spectral matrix norm (see, for example, \cite{Horn})
, the conditions $||A_rS||=||A_r||=\sigma_{\max}^{A_r}$ are met.
Given this fact, as well as the conditions \eqref{E1}, we get $\operatorname{rank}(\tilde A_c - A_r S) = n$ (see, for example \cite[Theorem 9.12]{Lawson}), and therefore due to the known properties of pseudo-inverse matrices of full column rank and residuals of pseudo-solutions 
\cite{Lawson}
there are equality
\begin{equation}
\label{DI}
(\tilde A_c - A_r S)^+(\tilde A_c - A_r S)=I_n, \;
\overset{\lower0.5em\hbox{$\smash{\scriptscriptstyle\frown}$}}{\Delta b^\top}(\tilde A_c - A_r S)= 0. 
\end{equation}
 Therefore, the conditions are met 
\begin{equation}
\label{D1}
P\left[ {\begin{array}{*{20}{c}}
  \tilde A_c - A_r S \\ 
  { - S} 
\end{array}} \right] 
=  
P_1(\tilde A_c - A_r S)
=
- I_n,\;
Q\left[ {\begin{array}{*{20}{c}}
  \tilde A_c - A_r S \\ 
  { - S} 
\end{array}} \right] 
= 
Q_1(\tilde A_c - A_r S)
=
I_n.
\end{equation}
At the same time,
\begin{equation}
\label{D2}
(P + Q)\left[ {\begin{array}{*{20}{c}}
  \tilde b_c + b_r \\ 
  0 
\end{array}} \right] 
= 
(P_1+Q_1)(\tilde b_c + b_r)
=
\left( {\alpha  + \beta } \right)||\Delta \overset{\lower0.5em\hbox{$\smash{\scriptscriptstyle\frown}$}}{b} |{|^2}S\overset{\lower0.5em\hbox{$\smash{\scriptscriptstyle\frown}$}}{x}  > 0.
\end{equation}

Now it remains to note that the conditions \eqref{D0}--\eqref{D2} are necessary
and sufficient conditions for the \emph{boundedness} of non-empty admissible domains of systems of linear inequalities 
\eqref{eq8} and \eqref{IE1}
{\cite[Problem 4.117]{Ashmanov}}, which in this case turn out to be not just convex polyhedral sets, but convex bounded polyhedra \cite{Ashmanov}.
\item
Let us construct the $(n\times m)$-matrix $G$ by the formula
\begin{equation}
\label{D3}
G = -S(\tilde A_c - A_r S)^+
+
\chi
S\overset{\lower0.5em\hbox{$\smash{\scriptscriptstyle\frown}$}}{x} \overset{\lower0.5em\hbox{$\smash{\scriptscriptstyle\frown}$}}{\Delta b^\top}.
\end{equation}

Due to \eqref{E4}
the scalar parameter $\chi$ can be chosen in such a way that it satisfies the conditions
\begin{equation}
\label{D4}
\max \left\{ {\mathop {\max }\limits_{i,j} {q_{ij}},0} \right\} \leqslant \chi < \frac{1}{{||\Delta \overset{\lower0.5em\hbox{$\smash{\scriptscriptstyle\frown}$}}{b} ||^2}}.
\end{equation}

Let us show that the condition $G\geqslant 0$ is satisfied. By virtue of the assumption \eqref{E3} and the above condition
$
S\overset{\lower0.5em\hbox{$\smash{\scriptscriptstyle\frown}$}}{x} > 0$ matrix elements
$
H = (h_{ij}) =
\operatorname{diag}
(S\overset{\lower0.5em\hbox{$\smash{\scriptscriptstyle\frown}$}}{x})^{-1}
\cdot
G
\cdot
\operatorname{diag}
(\Delta \overset{\lower0.5em\hbox{$\smash{\scriptscriptstyle\frown}$}}{b})^{-1}
$
have the same signs as the elements of the matrix $G$. But due to \eqref{E5} and \eqref{D3}
$h_{ij} = -q_{ij} +\chi$,
whence \eqref{D4} $H,G\geqslant 0$.

Note now that due to \eqref{DI} and \eqref{D4}
\[
G(\tilde A_c - A_r S) = -S,\; G(\tilde b_c + b_r)=
S\overset{\lower0.5em\hbox{$\smash{\scriptscriptstyle\frown}$}}{x}
(-1+\chi
||
\Delta \overset{\lower0.5em\hbox{$\smash{\scriptscriptstyle\frown}$}}{b}
||^2)
<0
,
\]
whence, by virtue of the Minkowski-Farkash theorem on consequences
{\cite[Theorem 4.7]{Ashmanov}}
there is such a number $\delta>0$ that the condition \eqref{eq10} will be satisfied.
\item
Note that
if the condition is met
\eqref{eq4}
then the system of linear inequalities
\begin{equation}
\label{eq9}
({{\tilde A}_c} + {A_r}S)x \leqslant {{\tilde b}_c} + {b_r}\;,\; Sx\geqslant 0
\end{equation}
is consistent.
This is true because the vector
$\hat x$
belongs to the set of admissible solutions of the system \eqref{eq9}.
Now let's notice that
the system of linear inequalities \eqref{eq8} is consistent by Lemma \ref{L_1}.
Besides,
\begin{equation}
\label{eq16}
\forall x\left| {Sx \geqslant 0}\right.,\forall \tilde S \ne S \Rightarrow - {A_r}Sx \leqslant - {A_r}\tilde Sx \leqslant {A_r}Sx.
\end{equation}

Taking into account compatibility
systems of linear inequalities \eqref{eq8} and \eqref{eq9},
ratios
\eqref{eq16}, lemme \ref{L_3}, \ref{L_4}
and conditions \eqref{eq10},
the systems below are \emph{joint},
the chain of consequences is valid (in which each subsequent system of linear inequalities is a consequence of the previous one):
\begin{equation*}
\left\{\!\!\! {\begin{array}{*{20}{c}}
  {({{\!\tilde A}_c} \!+\! {A_r}S)x\! \leqslant \! {{\tilde b}_c} \!+\! {b_r}} \\
  {Sx \geqslant 0}
\end{array}} \right. \!\!\!\!\Rightarrow\!\!
\left\{\!\!\! {\begin{array}{*{20}{c}}
  {({{\!\tilde A}_c} \!-\! {A_r}\tilde S)x\! \leqslant \! {{\tilde b}_c} \!+\! {b_r}} \\
  {Sx \geqslant 0}
\end{array}} \right. \!\!\!\!\Rightarrow\!\!
\left\{\!\!\! {\begin{array}{*{20}{c}}
  {({{\!\tilde A}_c} \!-\! {A_r}S)x\! \leqslant \! {{\tilde b}_c} \!+\! {b_r}} \\
  {Sx \geqslant 0}
\end{array}} \right.
\!\!\!\!\Rightarrow\!\!
Sx \!\geqslant\! {1}\delta,
\end{equation*}
and, finally,
all systems of linear inequalities of the form \eqref{eq11} are consistent and the system
$Sx \geqslant {1}\delta$
is a consequence of any of them.
\item
Note that 
\begin{equation}
\label{eq17}
\forall x\Rightarrow \left| A_c x - b_c\right| = \left|\right.\!\!\tilde{A_c} x - \tilde{b}_c\!\left|\right.. 
\end{equation}
By virtue of \eqref{eq17}, the system of inequalities \eqref{OP} can be written as 
\[
\left| {{A_c}x - {b_c}} \right| \leqslant {A_r}\left| x \right| + {b_r} \Leftrightarrow 
\left\{ {\begin{array}{*{20}{c}}
  {{A_c}x - {b_c} \leqslant {A_r}\left| x \right| + {b_r}} \\ 
  {{b_c} - {A_c}x \leqslant {A_r}\left| x \right| + {b_r}} 
\end{array}} \right..
\]
In its turn,
\begin{gather*}
  {{A_c}x - {b_c} \leqslant {A_r}\left| x\right| + {b_r}}
\Leftrightarrow
({{\tilde A}_c} - {A_r}\tilde S_j)x \leqslant {{\tilde b}_c} + {b_r},\\
  {{A_c}x - {b_c} \leqslant {A_r}\left| x\right| + {b_r}}
\Leftrightarrow
({{-\tilde A}_c} - {A_r}\tilde S_j)x \leqslant -{{\tilde b}_c} + {b_r},\\
j=1,\ldots,2^n,
\end{gather*}
where ${\tilde S_j}$ is one of the $2^n$ diagonal matrices of order $n$ with entries $ \pm 1$ on the diagonal.
But by virtue of the lemma \ref{L_4}, the corollaries
\[
\forall x\left|
\left\{ \begin{array}{*{20}{c}}
  (\tilde A_c-A_r\tilde S_j)x \leqslant \tilde b_c+b_r \\
  (-\tilde A_c-A_r\tilde S_j)x \leqslant -\tilde b_c+b_r \\
  j = 1,\dots,2^n
\end{array}
\right.
\Rightarrow Sx \geqslant {{1}}\delta . \right.
\]
Combining the calculations above, we get
\begin{equation}
\label{eq18}
\forall x\left|
\phantom{|^|_|}
\left| {{A_c}x - {b_c}} \right| \leqslant {A_r}\left| x\right| + {b_r}
\Rightarrow Sx \geqslant {{1}}\delta . \right.
\end{equation}
In turn, due to \eqref{eq18} and \eqref{eq3},
\[
\left| {{A_c}x - {b_c}} \right| \leqslant {A_r}\left| x \right| + {b_r}\Leftrightarrow
\left\{
\begin{array}{*{20}{c}}
(A_c - A_r S)x \leqslant b_c + b_r \\ 
(-A_c - A_r S)x \leqslant -b_c + b_r
\end{array}
\right.
\Leftrightarrow
\tilde{\mathbf{X}}
\equiv 
{\mathbf{\overset{\lower0.5em\hbox{$\smash{\scriptscriptstyle\frown}$}}{X} }}
\equiv
{\mathbf{X}}.
\]
But, as was shown in part 1 of the proof, the admissible region of the system of inequalities
$(\tilde A_c - A_r S)x\leqslant (\tilde b_c + b_r)$ is not empty and is a convex bounded polyhedron. Consequently,
in view of the foregoing, if the admissible area of the studied ISAE is not empty, it is a convex bounded polyhedron,
lying strictly inside the orthant determined by
the signs of the diagonal elements of the matrix $S$,
or, equivalently, by the signs of the elements of the vector $\hat x$
LS-solutions of the SLAE ${A_c}x \cong {b_c}$.
\end{enumerate}
\end{proof}

\section{Estimation of the error of an arbitrary solution of SLAE and its convergence to the normal solution of a hypothetical exact SLAE}

Let $A_0 x = b_0$ be a hypothetical 
\textquotedblleft
exact\textquotedblright
~consistent SLAE, where
$A_0\in\mathbb{R}^{m\times n}$, $b_0\in\mathbb{R}^{m}$, $x\in\mathbb{R}^{n}$,
$m>n$, $\operatorname{rank}A_0=n$. This system has a unique solution $x_0=A_0^+ b_0$, which is also a \textit{normal} solution (see, for example, \cite{Lawson}). Matrix $A_0$ and vectors $x_0$, $b_0$ are unknown.
Let $\tilde Ax=\tilde b$ be an approximate (not necessarily consistent) SLAE,
and the conditions
$| A_0 - \tilde A |\leqslant \Delta A$,
$| b_0 - \tilde b |\leqslant \Delta b$,
where $\tilde A,\Delta A\in\mathbb{R}^{m\times n}$ are known matrices,
$\tilde b,\Delta b\in\mathbb{R}^{m}$ are known vectors, $\Delta A,\Delta b \geqslant 0$,
$\operatorname{rank}\tilde A=n$,
$\sigma_{\min }^{\tilde A}>\sigma_{\max }^{\Delta A}$.

Consider the ISLAE view
\begin{equation}
\label{cnvr1}
Ax=b,\tilde A-\Delta A\leqslant A\leqslant \tilde A+\Delta A, \tilde b-\Delta b\leqslant b\leqslant \tilde b+\Delta b.
\end{equation}

The following
\begin{theorem}
ISLAE \eqref{cnvr1} is consistent. Moreover, for any of its solutions $\{A,b,x\}$
the following conditions are satisfied: $\operatorname{rank}A=n$, $x$ is the only solution of the system $Ax=b$ that is simultaneously a normal solution,
\begin{equation}
\label{cnvr2}
\begin{gathered}
||x-x_0||\leqslant \frac{||A_0^+||}{1-2||\Delta A||\cdot||A_0^+||}
\left(2||\Delta A||\cdot||x_0||+2||\Delta b||
\right)\leqslant
\\
\frac{2\alpha}{1 - 2\alpha\sigma_{\max }^{\Delta A}}
\left(\alpha\sigma_{\max }^{\Delta A}(||\tilde b||+||\Delta b||)
+ ||\Delta b||\right),
\end{gathered}
\end{equation}
where
\begin{gather}
\label{cnvr3}
\alpha = \frac{1}{{\sigma_{\min }^{\tilde A}}} + \frac{{\sqrt 2 }}{{\sigma_{\min }^{\tilde A} - \sigma_{\max }^{\Delta A}}} \cdot \frac{{\sigma_{\max }^{\Delta A}}}{{\sigma_{\min }^{\tilde A}}}, \\
\label{cnvr4}
\mathop {\lim }\limits_{||\Delta A||,||\Delta b|| \to 0} x = {x_0}.
\end{gather}
\end{theorem}
\begin{proof}
It is easy to verify that $\{A_0,b_0,x_0\}$ is a solution
\mbox{ISLAE} \eqref{cnvr1}.
The condition $\operatorname{rank}A =n$ follows from the conditions
$\operatorname{rank}\tilde A=n$,
$\sigma _{\min }^{\tilde A}>\sigma _{\max }^{\Delta A}$
by Theorem~\ref{TL}.
The uniqueness and normality of the solution $x$ follows from the condition
$\operatorname{rank}\tilde A=n$ \cite{Lawson,Horn}.
Since the system $Ax=b$ is compatible, $r=b-Ax=0$.
Due to the fulfillment of the conditions
$|\tilde A - A|\leqslant \Delta A$,
$|\tilde b - b|\leqslant \Delta b$
for any $\{A,b,x\}$ that are an ISLAE solution \eqref{cnvr1},
conditions are met
$|A_0 - A|\leqslant 2\Delta A$,
$|b_0 - b|\leqslant 2\Delta b$.
Therefore, by Theorem~\ref{TL},
\begin{equation}
\label{cnvr5}
||x-x_0||\leqslant \frac{||A_0^+||}{1-2||\Delta A||\cdot||A_0^+||} 
\left(2||\Delta A||\cdot||x_0||+2||\Delta b||
\right),
\end{equation}
whence the relation \eqref{cnvr4} immediately follows. It remains to get the top scores
unknown quantities $||A_0^+||$ and $||x_0||$. According to \cite{Lawson}
estimate $||\tilde A^+ - A_0^+||$ for the case
$\operatorname{rank}\tilde A = \operatorname{rank}A_0^+ = n$ has the form
$$
||\tilde A^+ - A_0^+||
\leqslant
\sqrt 2 \frac{{||\Delta A|| \cdot ||{{\tilde A}^ + }|{|^2}}}{{1 - ||\Delta A|| \cdot ||{{\tilde A}^ + }||}}
=
\frac{\sqrt 2 }{\sigma_{\min }^{\tilde A} - \sigma_{\max }^{\Delta A}} \cdot \frac{\sigma_{ \max }^{\Delta A}}{{\sigma_{\min }^{\tilde A}}},
$$
whence, by virtue of the obvious relation $||A_0^+||\leqslant ||\tilde A^+|| + ||\tilde A^+ - A_0^+||$, we have
$
||A_0^+||\leqslant \alpha,
$
where $\alpha$ is given by \eqref{cnvr3}.
Similarly, $||b_0||\leqslant||\tilde b||+||\Delta b||$.

Now note that the condition $x_0=A_0^+ b_0$ implies
$$
||x_0||\leqslant ||A_0^+||\cdot (||\tilde b||+||\Delta b||) \leqslant \alpha(||\tilde b||+||\Delta b||).
$$

Substituting into \eqref{cnvr5} the found upper bounds for the quantities $||A_0^+||$, $||x_0||$, $||b_0||$, we obtain the final inequality in \eqref{cnvr2}.
\end{proof}

\section{Numerical example}

As a numerical example, consider the inverse problem of chemical kinetics for an irreversible first-order reaction, which consists in determining two unknown parameters from experimental data: $c_0$ (the initial concentration of a substance) and $k$ (the reaction rate constant) in a kinetic model of the form
\begin{equation}
\label{K1}
c(t) = c_0 \exp (-kt),
\end{equation}
where $c(t)$ is the concentration of the substance at time $t$. The experimental data to be processed are taken from
\cite{Emanuel} and relate to the irreversible reaction of the decomposition of hexaphenylethane molecules into two molecules of the free radical triphenylmethyl:
\[
\mathrm{(C_6H_5)_3C - C(C_6H_5)_3 \to 2(C_6H_5)_3C},
\]
flowing at $0$ ${}^ \circ {\text{C}}$ in a mixture of 95 $\%$ toluene and 5 $\%$ aniline. The corresponding numerical values are presented in Table 1.
\begin{table}[h!]
\label{Tab1}
\caption{Experimental kinetics of hexaphenylethane decomposition}
\begin{tabular}{|p{0.8in}|p{0.4in}|p{0.4in}|p{0.4in}|p{0.4in}|p{0.4in}|p{0.4in}|p {0.4in}|p{0.4in}|p{0.4in}|} \hline
{\footnotesize $t^\text{ test}$, min}
 & 0 & 0.50 & 1.05 & 2.20 & 3.65 & 5.5 & 7.85 & 9.45 & 14.75 \\ \hline
{\footnotesize $c^\text{ test}(t)$, mol/l} & 0.1000 & 0.0934 & 0.0867 & 0.0733 & 0.0600 & 0.0465 & 0.0334 & 0.0265 & 0.0134 \\ \hline
\end{tabular}
\end{table}
Following the logic of \cite{Belov}, we will assume that the experimental data under study have an interval uncertainty
the following form:
\begin{gather*}
t_1 = 0,\; t_i = t_i^\text{ test}\pm \varepsilon_t,\; i=2,3,\ldots,9,\; \varepsilon_t=0.005, \\
c(t_i) = c^\text{ test}(t_i)\pm \varepsilon_c,\; i=1,2,\ldots,9,\; \varepsilon_c = 0.0005.
\end{gather*}

Transition from \eqref{K1} to a linearized model
$
\ln ({c(t)}) = \ln ({c_0}) - kt
$
allows you to form ISLAE with 9 interval equations and 2 unknowns,
coefficient matrices $A_c$, $A_r$ and vectors of the right side $b_c$, $b_r$ of which
have the following form
\[{A_c} = \left( {\begin{array}{*{20}{c}}
  1&{{t_1^\text{ test}}} \\
   \vdots & \vdots \\
  1&{{t_9^\text{ test}}}
\end{array}} \right),\;{A_r} = \left( {\begin{array}{*{20}{c}}
  0&0 \\
  0&{{\varepsilon _t}} \\
   \vdots & \vdots \\
  0&{{\varepsilon_t}}
\end{array}} \right),\;{b_c} = \left( {\begin{array}{*{20}{c}}
  {{\xi _1}} \\
   \vdots \\
  {{\xi _9}}
\end{array}} \right),\;{b_r} = \left( {\begin{array}{*{20}{c}}
  {{\zeta _1}}\\
   \vdots \\
  {{\zeta _9}}
\end{array}} \right),
\]
where
\[
{\xi _i} = 
\frac{{\ln\! \left( {c({t_i^\text{ test}}) - {\varepsilon _c}} \right) + 
\ln\! \left( {c({t_i^\text{ test}}) + 
{\varepsilon _c}} \right)}}{2},\;
{\zeta _i} = 
\frac{{\ln\! \left( {c({t_i^\text{ test}}) + 
{\varepsilon _c}} \right) - 
\ln\! \left( {c({t_i^\text{ test}}) - {\varepsilon _c}} \right)}}{2}.
\]

Calculations performed in the Mathcad\textsuperscript{\textregistered} environment give the following results:
\begin{gather*}
\hat x
\approx
\left( {\begin{array}{*{20}{r}}
   - 2.3088695 \\
   - 0.1374258
\end{array}} \right), \;
\overset{\lower0.5em\hbox{$\smash{\scriptscriptstyle\frown}$}}{x} 
\approx
\left( {\begin{array}{*{20}{r}}
-2.3146126\\
-0.1364464
\end{array}} \right), \;
S =
\left( {\begin{array}{*{20}{r}}
  { - 1}&0 \\ 
  0&{ - 1} 
\end{array}} \right), 
\\
\tilde A_c =\!
\left( {\begin{array}{*{20}{r}}
  1     &{   0.00} \\ 
  1     &{   0.50} \\ 
  1     &{   1.05} \\ 
  { - 1}&{ - 2.20} \\ 
  { - 1}&{ - 3.65} \\ 
  { - 1}&{ - 5.50} \\ 
  { - 1}&{ - 7.85} \\ 
  { - 1}&{ - 9.45} \\ 
  1     &{  14.75}
\end{array}} \right),  
\tilde b_c 
\approx 
\left( {\begin{array}{*{20}{r}}
  -2.302598 \\
  -2.370878 \\
  -2.445318 \\
   2.613218 \\
   2.813445 \\
   3.068360 \\
   3.399311 \\
   3.630789 \\
  -4.313197   
\end{array}} \right),
b_r
\approx 
\left( {\begin{array}{*{20}{r}}
0.005000\\
0.005353\\
0.005767\\
0.006821\\
0.008334\\
0.010753\\
0.014971\\
0.018870\\
0.037331
\end{array}} \right), 
\end{gather*}
\begin{gather*}
\Delta \overset{\lower0.5em\hbox{$\smash{\scriptscriptstyle\frown}$}}{b} 
\approx 
\left( {\begin{array}{*{20}{r}}
0.017015\\
0.017993\\
0.019013\\
0.005927\\
0.009819\\
0.014728\\
0.029248\\
0.046310\\
0.052012
\end{array}} \right),
A_r S\hat x 
\approx
\left( {\begin{array}{*{20}{l}}
0.000000\\
0.000687\\
0.000687\\
0.000687\\
0.000687\\
0.000687\\
0.000687\\
0.000687\\
0.000687\\
\end{array}} \right),\\
\sigma _{\min }^{{A_c}} 
\approx 
2.030051 >
\sigma _{\max }^{{A_r}}
\approx
0.014142, 
\operatorname{rank}A_c = \operatorname{rank}(\tilde A_c - A_r S) = 2,  \\
\gamma
\approx
0.040104,\;
||\Delta \overset{\lower0.5em\hbox{$\smash{\scriptscriptstyle\frown}$}}{b}||^2\cdot
\mathop {\max }\limits_{i,j} {q_{ij}} 
\approx
0.093881.
\end{gather*}

The presented numerical values testify to the fulfillment of the conditions \eqref{E1}--\eqref{E4} of the theorem \ref{TT2}.
The validity of the main statements of the theorem (the type and relative position of the admissible regions of the corresponding systems of inequalities) is shown graphically in Figure 1.

\begin{figure}[h!]
\center{\includegraphics[width=0.8\textwidth]{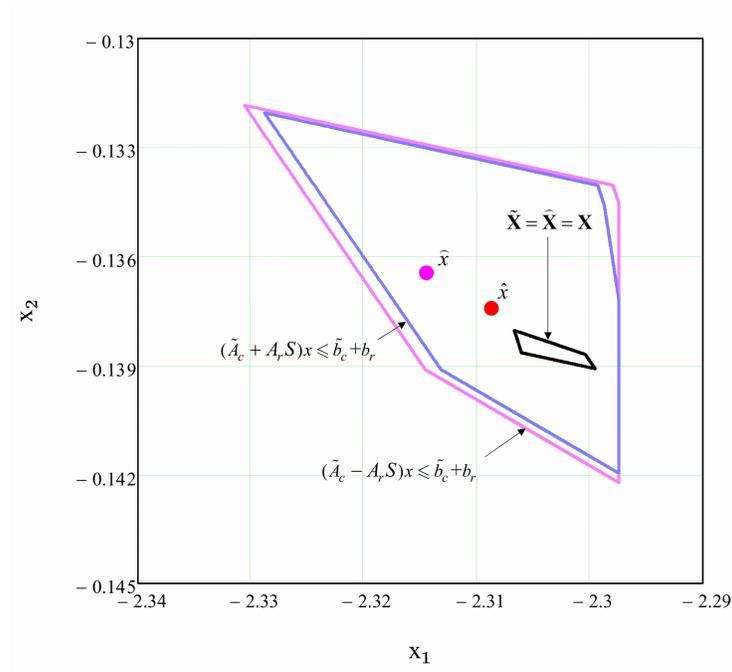}}
\caption{Illustration of fulfillment of the conditions of Theorem 2:
boundaries of the corresponding admissible domains, solutions of the corresponding LS-problems}
\label{F1}
\end{figure}

\section{Conclusion}

An attempt is made to bring together the theory and methods of interval systems of linear algebraic
equations with engineering practice of building linear models from experimental data with interval uncertainty. The results obtained (in the form of corresponding sufficient conditions) do not contradict
an intuitive requirement for initial data, which can be informally formulated as a requirement
the relative "smallness" of interval errors compared to the coefficients of the matrix $A_c$ and the vector $b_c$ of the "central" SLAE in combination with the requirement that the condition number of the matrix $A_c$ is not "not too high".

Some important questions are beyond the scope of this work. For example, discussion of numerical algorithms for finding least squares solutions and their residuals, determining the rank of matrices, calculating singular values of matrices.
This issue can be the subject of a separate further study, and at the same time, an extensive literature is devoted to it.
In the context of this article, we only note that the construction of LSM solutions and the corresponding residuals can be carried out by efficient, polynomial in complexity, finite-step or iterative methods,
and singular valuescan be calculated using efficient iterative algorithms with polynomial complexity. An overview of the corresponding algorithms with an estimate of their complexity can be found, for example, in the monograph
\cite{Golub}.

The same can be said about the problem of choosing an efficient numerical method for finding solutions to a system of linear inequalities, to which the problem of finding a solution to ISLAE has been reduced. Numerical methods of linear programming
continue to develop intensively, so the question raised may be the subject of a separate further study.

As another direction of further research, apparently, one can point to the search for sufficient conditions
"significance" of the coefficients of interval linear models based not on the least squares solution of the "central" SLAE,
and its pseudosolutions in other norms (${\ell _1}$, $\ell _\infty$ ).

%
%
%
%

\end{document}